\documentclass[12pt, reqno]{amsart}

\usepackage{amssymb, amsmath, amsthm,mathtools}
\usepackage{hyperref}
\usepackage[alphabetic,lite]{amsrefs}
\usepackage{amscd}   
\usepackage{fullpage}
\usepackage[all]{xy} 
\usepackage{faktor}
\usepackage{setspace}
\usepackage{enumitem}

\DeclareFontEncoding{OT2}{}{} 

\usepackage[usenames,dvipsnames]{color}

\newcommand{\ilim}{\mathop{\varprojlim}\limits}

\usepackage{tikz}
\usepackage{tikz-cd}

\newtheorem{lemma}{Lemma}[section]
\newtheorem{theorem}[lemma]{Theorem}
\newtheorem{proposition}[lemma]{Proposition}

\newtheorem{cor}[lemma]{Corollary}

\newtheorem{claim*}{Claim}

\newtheorem{defn}[lemma]{Definition}

\theoremstyle{definition}
\newtheorem{remark}[lemma]{Remark}

\newcommand{\PP}{{\mathbb P}}
\newcommand{\C}{{\mathbb C}}
\newcommand{\F}{{\mathbb F}}
\newcommand{\Q}{{\mathbb Q}}

\usepackage[mathscr]{euscript}

\DeclareMathOperator{\Aut}{Aut}
\DeclareMathOperator{\Gal}{Gal}

\DeclareMathOperator{\id}{id}

\DeclareMathOperator{\sgn}{sgn}

\newcommand{\into}{\hookrightarrow}

\numberwithin{equation}{section}
\numberwithin{table}{section}


\title{ Arithmetic Monodromy Groups of Dynamical Belyi maps}
\author{\"{O}zlem Ejder}
\address{Bo\u{g}azi\c{c}i  University, Istanbul}
\email{ozlem.ejder@boun.edu.tr}

\begin{document}

\begin{abstract}
We consider a large family of dynamical Belyi maps of arbitrary degree and study
the arithmetic monodromy groups attached to the iterates of such maps.  Building on the results of Bouw-Ejder-Karemaker on the geometric monodromy groups of these maps, we show that the quotient of the arithmetic monodromy group by the geometric monodromy group has order either $1$ or $2$. Prior to this article, a result of this kind was only known for quadratic maps (Pink) and a few examples in degree $3$.
 \end{abstract}
\maketitle

\section{Introduction}

Let $f:\PP^1_k \to \PP^1_k$ be a rational map of degree $d$ defined over a number field $k$. For each $n\geq 1$, define the $n$-th iterate of $f$ by $f^n=f\circ\ldots\circ f$. It was Odoni \cite{Odoni85} who first studied the Galois theory of the iterates of $f$ mainly for its applications in dynamical systems. Assume $f$ is postcritically finite (PCF), i.e., the orbit of each critical point is finite. Let $P=\{f^n(x) \in \PP^1(k) : x \text{ is a critical point of } f \text{ and } n\geq 1\}$. Then the iterates $f^n$ are unbranched outside the finite set $P$. For a point $x_0  \in \PP^1_k\backslash P$, one can construct a tree $T$ whose leaves are the points in $f^{-n}(x_0)$ for $n\geq 1$. We obtain a representation of the \'etale fundamental group of $\PP^1_k\backslash P$ (resp., $\PP^1_{\bar{k}}\backslash P$) inside the automorphism group of the tree $T$. We call the image of such map the arithmetic (resp., geometric) monodromy group $G^{\text{arith}}$ (resp.,  $G^{\text{geom}}$) of $f$. See Section~\ref{sec:mon} for details. For PCF maps, the geometric fundamental group is topologically finitely generated.

Pink \cite{Pinkpolyn} \cite{Pinkrational} has studied the case $d=2$ extensively. He showed that the arithmetic and geometric monodromy groups of the quadratic PCF maps are determined only by the combinatorial data of the postcritical orbit $P$. Moreover he described the quotient group $G^{\text{arith}}/G^{\text{geom}}$ for quadratic polynomials and quadratic morphisms with infinite postcritical orbit $P$. 

Similarly the article \cite{BEK} determines the Galois groups attached to a large class of Belyi maps of any degree $d$. A rational map $f: \PP^1_k \to \PP^1_k$ is called a \emph{Belyi map} if it is branched exactly over $\{0,1,\infty\}$. The authors of \cite{BEK} show that the geometric monodromy group is again only determined by their combinatorial type for Belyi maps $f: \PP^1_k \to \PP^1_k$ with exactly three ramification points; $0,1,\infty$ which are fixed by $f$. We call these maps normalized, single cycle genus zero Belyi maps. A dynamical Belyi map is a Belyi map $f: \PP^1_k \to \PP^1_k$ where $f(\{0,1,\infty\}) \subset \{0,1,\infty\}$. All maps considered in \cite{BEK} are all normalized dynamical Belyi maps which are all PCF. 

This article is essentially a sequel to \cite{BEK} where a condition is described for the quotient of the arithmetic and the geometric monodromy group of normalized single cycle dynamical Belyi maps to be trivial. This is done by introducing a product discriminant. In this article, we use group theoretical tools to describe the normalizer of the geometric monodromy group inside the automorphism group of the tree $T$. This helps us bound the size of the quotient group $G^{\text{arith}}/G^{\text{geom}}$ for general single cycle dynamical Belyi maps. In particular we prove Theorem~\ref{thm:main} which states that the quotient $G^{\text{arith}}/G^{\text{geom}}$has order at most $2$ for all but finitely many of these maps. 

See also the article \cite{Jones-Manes} for the Galois groups of iterates of quadratic maps, \cite{AHMAD2021} for a similar result for $f(x)=x^2-1$, \cite{BFHJY} for an example in degree $3$, and \cite{Jones-survey} for a fantastic survey on the subject.


\section{Automorphism Group of $T$}

Let $d \geq 3$. Let $T$ be the infinite regular $d$-ary tree whose vertices are the finite words over the alphabet $\{0,\ldots,d-1\}$. For any integer $n\geq 1$, we let $T_n$ denote the finite rooted subtree whose vertices are the words of length at most $n$. We will call the set of words of length $n$  the level $n$ of $T$. We will use the notation $W:=\Aut(T)$ and $W_n:=\Aut(T_n)$. 

We embed  $W^d:= W \times \ldots \times W$ into $W$ by identifying the complete subtrees rooted at level one of the tree $T$ with $T$ itself. The image of the embedding $W^d \into W$ is given by 
the set of automorphisms acting trivially on the first level. 
The exact sequence $1 \to W^d  \to W \to W_1 \to 1$ splits and gives the semi direct product 
\[ W \simeq W^d \rtimes S_d \text{ and }  W_n \simeq W_{n-1}^d \rtimes S_d. \]
In other words, $W$ and $W_n$ have a wreath product structure:
\[W \simeq W \wr S_d \text{ and } W_n \simeq W_{n-1} \wr S_d  \]
for $n\geq 2$. 

We denote an element of $W$ (resp. $W_n$) by $(x_1,\ldots x_d)\tau$ where $x_i$ are in $W$ (resp. $W_{n-1}$) and $\tau \in S_d$. 
We have the following relations in $W$:
\begin{equation}\label{eq:relations} \begin{aligned} (x_1,\ldots, x_d) (y_1,\ldots, y_d)&= (x_1y_1,\ldots,x_d y_d) \\
                   \tau(x_1,\ldots,x_d) &= (x_{\tau^{-1}(1)},\ldots,x_{\tau^{-1}(d)}) \tau \\
 \end{aligned}
\end{equation}

We embed $S_d$ into $W$ by the map $\tau \mapsto (-,\ldots,-)\tau$. We denote the automorphism $(-,\ldots,-)\tau$ simply by $\tau$ in $W$. Here $-$ denotes the identity. For every $m\leq n$, we write $\pi_m$ for the natural projection
\[
\pi_m: W_n \to W_m,
\]
which corresponds to restricting the action of an element of
$W_n$ to the subtree $T_m$ consisting of the levels $0, 1, \ldots, m$. Abusing the notation let $\pi_m: W \to W_m$ also denote the natural projection onto the finite level $m$. We denote the image of an element $w$ under $\pi_m$ as $w_{|T_m}$.

Let $G$ be a subgroup of $W$. For each $n\geq 1$, we define $G_n:=\pi_n(G) \subset W_n$. Next we define some subgroups of $W$ which will be essential in Theorem~\ref{thm:bek} where we describe the geometric monodromy groups of dynamical Belyi maps. To define these groups, we first need to define a \emph{product sign} map on $W$.

\begin{defn}\label{def:sgn2}
  Define $\sgn_2: W_2 \to \{\pm 1\}$ by
setting
 \begin{equation}\label{eq:wrsgn}
 \sgn_2((x_1, \ldots, x_d)\tau) = \sgn(\tau) \prod_{i=1}^d
 \sgn(x_i).
 \end{equation}
Here $\sgn$ is the usual sign on $W_1$ via the identification $
W_1\simeq S_d$ induced by the choice of labeling of the vertices.
We define
\begin{equation}\label{eq:sgn2}
\sgn_2:=\sgn_2\circ \pi_2: W \to \{\pm 1\} 
\end{equation}

\end{defn}
Note that we abuse the notation and denote the maps $W \to W_2$ and $W_n \to W_2$ both by $\pi_2$.

\begin{defn}\label{def:En} \hfill
\begin{enumerate}
\item Define the subgroup $E_n\subseteq W_n$ by
\[
E_n=\begin{cases} W_1 &\text{ if }n=1,\\ (E_{n-1}\wr
E_1)\cap \ker(\sgn_2)\subseteq  W_n
&\text{ otherwise}
\end{cases}
\]  and
\[  E:=\ilim_{n} E_n \subseteq W \]

\item Define the subgroup $U_n\subseteq W_n$ as the $n$-fold iterated wreath product of $A_d$ and
 \[ U:=\ilim_{n} U_n \subseteq W \]
 \end{enumerate}
\end{defn}

\begin{remark}
Note that it follows from Definition~\ref{def:En} that  $E \simeq (E \wr S_d) \cap \ker(\sgn_2)$. Let $w=(x_1,\ldots,x_d)\tau \in E$. By definition $w_{|T_n}$ is in $E_n$ for any $n\geq 1$. Hence $w_{|{T_2}} \in \ker(\sgn_2)$ and ${x_i}_{|T_{n-1}}$ is in $E_{n-1}$ for all $i$ which implies that $x_i \in E$ for all $i\geq 1$. Conversely, let $w=(x_1,\ldots,x_d)\tau \in W$. If $x_i \in E$ for all $i\geq 1$ and $w_{|{T_2}} \in \ker(\sgn_2)$, then $w_{|T_n} \in E_n$ and hence $w$ is in $E$.

Similarly, $w \in U$ if and only if $x_i \in U$ for each $i$ and $\tau \in A_d$.

\end{remark}


\section{Belyi Maps}\label{sec:belyi}
A (\emph{genus zero}) \emph{Belyi map} is a rational map $f: \PP^1_{\C} \to \PP^1_{\C}$ such that $f$ is branched exactly over $x_1=0, x_2=1,$ and $x_3=\infty$. It is called a \emph{dynamical Belyi} map if $f(\{0,1,\infty\})\subset \{0,1,\infty\}$. Hence the iterates of a dynamical Belyi map are also dynamical Belyi maps.

A Belyi map is called \emph{single cycle} if there is a unique ramification point over each of the three branch points. It is called \emph{normalized} if $f(0)=0$, $f(1)=1$, and $f(\infty)=\infty$. Hence a normalized Belyi map is dynamical.

The \emph{combinatorial type} of a single cycle Belyi map is the tuple $(d; e_1,e_2,e_3)$ where $d$ denotes the degree of $f$ and $e_i$ denotes the ramification index of the unique ramification point above each $x_i$. 
An abstract type is a tuple $(d; e_1, e_2, e_3)$ such that $2\leq e_1\leq e_2\leq e_3\leq d$ and $e_1+e_2+e_3=2d+1$. For each abstract type $\underline{C}$, there exists a unique normalized Belyi map of type $\underline{C}$ which can be defined over $\Q$. See \cite[Proposition 1]{ABEGKM} for a proof. 

Notice that a normalized dynamical Belyi map is postcritically finite (PCF), i.e., the orbit of each critical point is finite. In this paper, we will focus on the genus zero, single cycle normalized Belyi maps which form a large class of dynamical Belyi maps.


\section{Monodromy Groups of Dynamical Belyi Maps}\label{sec:mon}

Let $k$ be a number field. Fix an algebraic closure $\bar{k}$ of $k$. Let $P=\{0,1,\infty\}$ and 
 let $f$ be a dynamical Belyi map of combinatorial type $\underline{C}$ defined over $k$. 

Let $x_0 \in \PP^1_k(k)\backslash P$. Then each $f^n$ is a connected unramified covering of $\PP^1_k \backslash P$, hence it is determined  by the monodromy action of $\pi_1^{\acute{e}t}(\PP_k^1 \backslash P, x_0)$ on $f^{-n}(x_0)$) up to isomorphism. Let $T_{x_0}$ be the tree defined as follows: it is rooted at $x_0$, the leaves of $T_{x_0}$ are the points of $f^{-n}(x_0)$ for all $n\geq 1$, and the two leaves $p,q$ are connected if $f(p)=q$.  Varying $n$, associated monodromy defines a representation 
\begin{equation}\label{eq:rho}
 \rho: \pi_1^{\acute{e}t}(\PP_k^1 \backslash P, x_0) \to \Aut(T_{x_0})
\end{equation} 
whose image we call the \emph{arithmetic monodromy group} $G^{\text{arith}}(f)$ of $f$. One can also study this representation over $\bar{k}$ and obtain 
\[ \pi_1^{\acute{e}t}(\PP_{\bar{k}}^1 \backslash P, x_0) \to \Aut(T_{x_0}).
\]
 We call the image of the map in this case the \emph{geometric monodromy group} $G^{\text{geom}}(f)$. We note that these monodromy groups are unique up to conjugation by the elements of $\Aut(T_{x_0})$. The arithmetic and the geometric monodromy groups of $f$ fit into an exact sequence as follows:
 \begin{equation}\label{eq:exact}
   \begin{tikzcd}
1\arrow{r}  &  \pi_1^{\acute{e}t}(\PP_{\bar{k}}^1 \backslash P, x_0) \arrow{r} \arrow{d}& \pi_1^{\acute{e}t}(\PP_k^1 \backslash P, x_0) \arrow{r} \arrow{d} &\Gal(\bar{k}/ k) \arrow{r} \arrow{d} &1  \\
   1 \arrow{r} &   G^{\text{geom}}(f) \arrow{r} &G^{\text{arith}}(f) \arrow{r}  &\Gal(L /k) \arrow{r} &1 \\
     \end{tikzcd}
  \end{equation}
for some field extension $L$ of $k$. Determining this field $L$ or its degree is a fundamental problem. The groups $G^{\text{geom}} $ and $G^{\text{arith}}$ are profinite groups and they are embedded into $\Aut(T_{x_0})$ by construction. 

Let $x_1,\ldots,x_d$ denote the points in $f^{-1}(x_0)$. Then $f^{-(n+1)}(x_0)=f^{-n}(x_1) \cup \ldots \cup f^{-n}(x_d)$. Hence after deleting the root $x_0$,  the tree $T_{x_0}$ decomposes into the $d$ regular trees $T_{x_1},\ldots,T_{x_d}$. Let $\tilde{P}$ denote the set $f^{-1}(P)$. Then by functoriality, $f: \PP^1_k \backslash \tilde{P} \to \PP^1_k \backslash P$ induces a map on the fundamental groups: 
\begin{equation}\label{eq:f*}
f^i_{*}: \pi_1^{\acute{e}t}(\PP_k^1 \backslash \tilde{P}, x_i) \to \pi_1^{\acute{e}t}(\PP_k^1 \backslash P, x_0).
 \end{equation} for any $i \in \{1,\ldots,d\}$.
 Here we use the notation $f^i_*$ to specify the base point $x_i$ of $\pi_1^{\acute{e}t}(\PP_k^1 \backslash \tilde{P}, x_i)$. This should not be confused with the $i$'th iteration of $f$.
 Similarly, by functoriality, the inclusion map $\PP_k^1\backslash \tilde{P} \hookrightarrow  \PP_k^1 \backslash P$ induces the surjective map
 \begin{equation}\label{eq:id}
id_{*}: \pi_1^{\acute{e}t}(\PP_k^1 \backslash \tilde{P}, x_i) \to \pi_1^{\acute{e}t}(\PP_k^1 \backslash P, x_i).
 \end{equation}
 for any $i \in \{1,\ldots,d\}$.
The action of $\pi_1^{\acute{e}t}(\PP_k^1 \backslash \tilde{P}, x_i) $ on $T_{x_1}$ through $f^i_{*}$ coincides with its natural action on $T_{x_i}$. That is, the image of the composition of $f^i_{*}$ with \eqref{eq:rho} is exactly the automorphisms in the image of $\rho$ that fix $x_i$. However these automorphisms do not have to fix any of the other $x_j$ for $j\neq i $. Since we would like to describe the action on the subtrees $T_{x_i}$ for all $i$,
we construct the following subgroup:
\[ 
F_k:=\{ (\gamma_1,\ldots,\gamma_d) \in \prod_{i=1}^d \pi_1^{\acute{e}t}(\PP_k^1 \backslash \tilde{P}, x_i) : \rho\circ f^i_{*} (\gamma_i)=\rho\circ f^j_{*}(\gamma_j )\text{ for any } i,j \in \{1,\ldots,d\} \}.
\]
Now an automorphism in the image of the composition $F_k \to \pi_1^{\acute{e}t}(\PP_k^1 \backslash \tilde{P}, x_i)$ (for any $i$) and $\rho \circ {f^i}_{*}$ fixes $x_i$ for all $i$. This action coincides with the natural action of $F_k$ composed with \eqref{eq:id} on $ \sqcup_{i=1}^d T_{x_i}$. This is given in the left upper square of Diagram~\eqref{diagram:mon}. The first vertical isomorphism on the left is obtained by a change of base point from $x_1$ to $x_0$.  We identify the trees $T_{x_0},T_{x_1},\ldots,T_{x_d}$ with the regular $d$-ary tree $T$ introduced earlier. We can do this exactly because $x_0 \not\in P$.

\begin{equation}\label{diagram:mon}
   \begin{tikzcd}
\pi_1^{\acute{e}t}(\PP_k^1 \backslash P, x_0) \arrow{r}{\rho}      & \Aut(T_{x_0})  \arrow{r}{\simeq}                                 & W \\
     F_k       \arrow{r}    \arrow{u}{f_{*}} \arrow{d}{pr_1}                  &   \Aut(T_{x_1}) \times \ldots \times \Aut(T_{x_d})  \arrow{dd}{pr_1} \arrow[hookrightarrow]{u} \arrow{r}{\simeq} & W^d  \arrow[hookrightarrow]{u} \arrow{dd}  \\
  \pi_1^{\acute{e}t}(\PP_k^1 \backslash \tilde{P}, x_1) \arrow{d}{\simeq} \arrow{dr}    \\
      \pi_1^{\acute{e}t}(\PP_k^1 \backslash \tilde{P}, x_0)  \arrow[two heads]{d}{id_{*}} \arrow{dr}        &  \Aut(T_{x_1}  )  \arrow{d}{\simeq}   \arrow{r}{\simeq}& W \ar[equal]{d}                            \\
       \pi_1^{\acute{e}t}(\PP_k^1 \backslash P, x_0)         \arrow{r}                                        &       \Aut(T_{x_0})     \arrow{r}{\simeq}& W                                                     
 \end{tikzcd}
\end{equation}

We also note here that a discussion of this kind is given in \cite[pg~20]{Pinkpolyn} for rational maps of degree $2$. We generalize Pink's argument to any $d$ here. In the case $d=2$, if an automorphism fixes $x_1$, then it also has to fix $x_2$. Hence there is no need to define a subgroup $F_k$ and hence Pink only uses $ \pi_1^{\acute{e}t}(\PP_k^1 \backslash \tilde{P}, x_1)$.

By the construction of $G^{\text{arith}}$, \eqref{diagram:mon} induces a commutative diagram

\begin{equation}\label{diagram:Gar}
   \begin{tikzcd}
\pi_1^{\acute{e}t}(\PP_k^1 \backslash P, x_0) \arrow{r}{\rho}     &G^{\text{arith}}  \arrow[hookrightarrow]{r}                                & W \\
     F_k       \arrow{r}    \arrow{u}{f_{*}} \arrow{d}{pr_1}                  & G^{\text{arith}} \cap W^d  \arrow{d}{pr_1} \arrow[hookrightarrow]{u} \arrow{r}& W^d  \arrow[hookrightarrow]{u} \arrow{d}  \\
       \pi_1^{\acute{e}t}(\PP_k^1 \backslash P, x_0)         \arrow[two heads]{r}                                        &      G^{\text{arith}}     \arrow[hookrightarrow]{r} & W                                                     
 \end{tikzcd}
\end{equation}

We obtain a similar diagram for $G^{\text{geom}}$ when we replace $k$ by $\bar{k}$. 
\begin{equation}\label{diagram:Ggeom}
   \begin{tikzcd}
\pi_1^{\acute{e}t}(\PP_{\bar{k}}^1 \backslash P, x_0) \arrow{r}{\rho}     &G^{\text{geom}}  \arrow[hookrightarrow]{r}                                & W \\
     F_{\bar{k}}       \arrow{r}    \arrow{u}{f_{*}} \arrow{d}{pr_1}                  & G^{\text{geom}} \cap W^d  \arrow{d}{pr_1} \arrow[hookrightarrow]{u} \arrow{r}& W^d  \arrow[hookrightarrow]{u} \arrow{d}  \\
       \pi_1^{\acute{e}t}(\PP_{\bar{k}}^1 \backslash P, x_0)         \arrow[two heads]{r}                                        &      G^{\text{geom}}     \arrow[hookrightarrow]{r} & W                                                     
 \end{tikzcd}
\end{equation}

 Diagrams~\eqref{diagram:Gar} and \eqref{diagram:Ggeom} will be used in the proof of the main theorem in Section~\ref{sec:6}. For single cycle normalized Belyi maps, the group $G^{\text{geom}}$ is determined in \cite{BEK}.
Remember the groups $E$ and $U$ are defined in Definition~\ref{def:En}. We first give an existing result on the geometric monodromy groups of Belyi maps at level one.

\begin{theorem}\cite[Theorem~5.3]{liuosserman}\label{Lui-Osserman} 
Let $f$ be a normalized Belyi map of type $\underline{C}=(d; e_1, e_2,
e_3)\notin\{(4;3,3,3), (6;4,4,5)\}$. 
\begin{enumerate}

\item  If at least one of the $e_j$ is even, then $G_1^{\text{geom}}(f) \simeq S_d$. 
\item If all $e_j$ are odd, then $G_1^{\text{geom}}(f)  \simeq A_d.$ 
\end{enumerate}
\end{theorem}

The next result describes the geometric monodromy group $G^{\text{geom}}$ for single cycle, normalized dynamical Belyi maps.

\begin{theorem}\cite[Theorem 2.3.1]{BEK}\label{thm:bek} 
Let $f$ be a normalized Belyi map of type $\underline{C}=(d; e_1, e_2,
e_3)\notin\{(4;3,3,3), (6;4,4,5)\}$.  
\begin{enumerate}
\item Assume that at least one of the $e_j$ is even.  Then
$G^{\text{geom}}(f)  \simeq E$.
\item Assume that $e_i$ are all odd. Then 
$G^{\text{geom}}(f)\simeq U$.
\end{enumerate}
\end{theorem}

In the same paper, a criteria for the triviality of the quotient $G^{\text{arith}} /G^{\text{geom}}$ is given. See \cite[Corollary~2.4.6]{BEK}. In this article, we prove that the order of the quotient group $G^{\text{arith}} /G^{\text{geom}}$ is either $1$ or $2$. As seen in \eqref{eq:exact}, $G^{\text{arith}}$ can be seen as a subgroup of the normalizer of $G^{\text{geom}}$ in $W$. We will study the normalizer of the subgroups $E$ and $U$ in the next section.

 
\section{Normalizer of $E$ and $U$ inside $W$}
Let $G$ be either $E$ or $U$. Let $N(G)$ denote the normalizer of $G$ in $W$. Given an element $x \in N(G)$, we denote its image in the quotient $N/G$ by $x \pmod G$.
\begin{proposition}\label{prop:ncomp}
 Let $G$ be one of the groups $E$ or $U$. Let $x=(x_1,\ldots,x_d)\tau \in N(G)$. Then  \hfill
\begin{enumerate}
 \item $x_i$ is in $N(G)$ for all $1\leq i\leq d$.
\item $x_i x_j^{-1}$ is in $G$ for all $1\leq i,j \leq d$.
\end{enumerate}
\end{proposition}
\begin{proof}
Let $x=(x_1,\ldots,x_d)\tau \in N(G)$ and let $g \in G$.  We will show that $x_igx_i^{-1}$ is in $G$ for all $1\leq i\leq d$.
We will assume $i=1$ for simplicity and let $j= \tau^{-1}(1)$. Let $y=(y_1,\ldots,y_d)$ be an element of $W$ such that
$y_j=g$ and $y_i=\id$ for $i\neq j$. Similarly, let $y'=(z_1,\ldots,z_d) \in W$ with $z_1=z_j=g$ and $z_i=\id$ for $i\not\in \{ 1,j\}$. Since $g$ is in $G$,
either $y$ or $y'$ is in $G$. This only depends on the sign of $g_{|T_1}$ in $S_d$. Assume $y\in G$.  We compute
 \[ 
  \begin{aligned}
  xyx^{-1}=&(x_1,\ldots ,x_d)\tau(y_1,\ldots, y_d)\tau^{-1}(x_1^{-1},\ldots x_d^{-1} )\\
                        =&(x_1y_{\tau^{-1}(1)}x_1^{-1}, \ldots, x_dy_{\tau^{-1}(d)}x_d^{-1}).   \\
  \end{aligned}
\] 
Since $x$ is in $N(G)$, $ xyx^{-1} \in G$, and by the construction of $y$, 
\[ x_1y_{\tau^{-1}(1)}x_1^{-1}=x_1y_{j}x_1^{-1}=x_1gx_1^{-1} \]
is also in $G$ and this finishes the proof of the first claim. Notice that if $y\not\in G$, then one can use $y'$ instead. 

For the second claim, fix $i$ and $j$ in $\{1,\ldots,d\}$. We take an element $y=\sigma \in G$ where $\sigma$ is an even cycle such that  $\tau\sigma\tau^{-1}(j)=i$. Then 
 \[ 
  \begin{aligned}
  xyx^{-1}=&(x_1,\ldots,x_d)\tau\sigma\tau^{-1}(x_1^{-1},\ldots,x_d^{-1}) \\
               =&(x_1x^{-1}_{\tau'^{-1}(1)},\ldots,x_dx^{-1} _{\tau'^{-1}(d)})
   \end{aligned}
 \]
 where $\tau'=\tau\sigma\tau^{-1}$. Since $\tau'^{-1}(i)=j$ and $xyx^{-1} \in G$, we have $x_ix_j^{-1}$ is in $G$.
 \end{proof}
\begin{remark}
Note that one can replace $G$ and $N(G)$ by $G_n$ and $N(G)_n$ in Proposition~\ref{prop:ncomp} and its proof.
\end{remark}
Let $\sigma=(1 2) \in S_d$. Remember that we embed $S_d$ into $W$ by $\tau \mapsto (-,\ldots,-)\tau$ for any $\tau \in S_d$.
\begin{defn}
For each $i\geq 1$, define an element of $W$ as follows:
\[ w_i:=
\begin{cases} \sigma & \text{ for } i=1\\
                     (w_{i-1},\ldots,w_{i-1}) &\text{ for } i\geq 2.
\end{cases}
\]
\end{defn}

Let $N(E)$ and $N(U)$ denote the normalizer of $E$ and $U$ in $W$ respectively.  Remember that we denote the restriction of $N(E)$ and $N(U)$ to level $n$ by $N(E)_n$ and $N(U)_n$.

\begin{lemma}\label{lem:w}  \hfill
\begin{enumerate}
\item For $k \geq 1$, the automorphism $w_k$ has order $2$.
\item
For $k \geq 1$, the automorphism $w_k$ is in $N(E) \backslash E$ (resp., $N(U)\backslash U$). 
\end{enumerate}
\end{lemma}
\begin{proof}
The first part follows by induction on $k$ using the equality $w_k^2=(w_{k-1}^2,\ldots,w_{k-1}^2)$. We prove the second part by induction on $k$ as well. Let $k=1$ and let $g=(g_1, \ldots, g_d)\tau \in W$. We compute 
 \[
 w_1gw_1^{-1}=\sigma(g_1,\ldots, g_d)\tau\sigma^{-1}=(g_2, g_1,g_3,\ldots, g_d)\sigma\tau \sigma^{-1}. 
 \] 
 Assume $g$ is in $E$ (resp., in $U$), then each $g_i$ are in $E$ (resp., in $U$). Since the sign of a permutation is invariant under conjugation and $g_i \in E$ (rep., $\in U$) , we have $w_1gw_1^{-1} \in E$ (resp., $ \in U$). This proves that 
$w_1$ is in $N(E)$ (resp., in $N(U)$). Assume $w_{k-1}$ is in $N(E)$ (resp., in $N(U)$). Then 
\[ 
  \begin{aligned}
  w_k gw_k^{-1}=&(w_{k-1},\ldots ,w_{k-1})(g_1,\ldots, g_d)\tau(w_{k-1}^{-1},\ldots w_{k-1}^{-1} )\\
                        =&(w_{k-1}g_1w_{k-1}^{-1}, \ldots,w_{k-1} g_dw_{k-1}^{-1})\tau.   \\
  \end{aligned}
\]  
By the induction hypothesis $w_{k-1}g_iw_{k-1}^{-1}$ is in $E$ (resp., in $U$) for all $i$ and $\sgn_2( w_k gw_k^{-1})=\sgn_2(g)=1$, hence $w_k$ is in $N(E)$. Since $\sgn_2(w_1)=-1$ (resp., $\sgn(\sigma) \neq 1$ ), $w_1$ is not in $E$ (resp., not in $U$) and similarly $w_k \not\in E$ (resp., $\not\in U$) since $w_1\not\in E$ (resp., $\not\in U$).
\end{proof}

\begin{lemma}\label{lem:wcommute} 
We have $w_i  w_j =w_j w_i $ for all $i,j \geq 1$.
\end{lemma}
\begin{proof}
We first observe that $w_1$ commutes with $w_i$ for any $i\geq 1$. This follows from the relations given in Equation~\eqref{eq:relations}. Assume $w_i$ and $w_j$ commute for all $i,j \leq n$. Let $i>1$.
Then \[ w_iw_{n+1}=(w_{i-1},\ldots,w_{i-1})(w_n,\ldots,w_n)= (w_{i-1}w_n,\ldots,w_{i-1}w_n). \]
By our assumption $w_{i-1}w_n=w_nw_{i-1}$ and hence $w_iw_{n+1}=w_{n+1}w_i$.

\end{proof}

\begin{defn}\label{def:isom}\hfill
\begin{enumerate}
\item Let $n\geq 2$. Define $ \varphi_n: \prod_{i=1}^{n-1} \mathbb{F}_2  \to  N(E)_n/E_n$ such that   
       \[ 
       \varphi((k_1,\ldots,k_{n-1}))=\prod_{i=1}^{n-1}{{w_i}_{|T_n}}^{k_i} \pmod {E_n}  
        \]         
\item Let $n\geq 1$. Define $\phi_n: \prod_{i=1}^{n} \mathbb{F}_2  \to  N(U)_n/U_n$ such that
 \[ \phi((k_1,\ldots,k_{n}))=\prod_{i=1}^{n}{{w_i}_{|T_n}}^{k_i} \pmod {U_n}.
  \]
\end{enumerate}
\end{defn}

\begin{lemma} The map $\varphi_n$ (resp., $\phi_n$) is a well defined homomorphism for any $n\geq 2$ (resp., for any $n\geq 1$).
\end{lemma}
\begin{proof}
By Lemma~\ref{lem:w}(2), the order of $w_i$ is $2$ for any $i\geq 1$, hence the maps $\varphi_n$ and $\phi$ are well-defined. 
By Lemma~\ref{lem:wcommute}, they are both homomorphisms.
\end{proof}

\begin{proposition}
The homomorphism $\varphi_n$ (resp., $\phi_n$) is injective for any $n\geq 2$ (resp., for any $n\geq 1$).
\end{proposition}
\begin{proof}
 We will do induction on $n$ to show that $\varphi_n$ is injective. If ${w_1}^{k_1}_{|T_2} =\id \pmod {E_2}$, then  ${w_1}_{|T_2}^{k_1} $ is in $E_2$. Since $\sgn_2(w_1)=-1$, we have $k_1=0$ and hence $\varphi_2$ is injective. Similarly $\phi_1$ is injective since ${w_1}_{|T_2} $ is not in $U_1=A_d$. 
 
Assume that $\varphi_{n}$ (resp., $\phi_n$) is injective for some $n\geq 2$ (resp., $n\geq 1$). 
Let $\varphi_{n+1}(k_1,\ldots,k_{n})=\prod_{i=1}^{n}{w_{i}}_{|T_{n+1}}^{k_i}=\id \pmod {E_{n+1}}$.  We compute that
 \[ \begin{aligned}
    \prod_{i=1}^{n}{w_i}_{|T_{n+1}}^{k_i}&={w_{1}}_{|T_{n+1}}^{k_1}\prod_{i=2}^{n}({w_{i-1}}_{|T_{n}}^{k_i}, \ldots, {w_{i-1}}_{|T_{n}}^{k_i} ) \\
                                                       &={w_{1}}_{|T_{n+1}}^{k_1}(\prod_{i=1}^{n-1}{w_{i}}_{|T_{n}}^{k_{i+1}}, \ldots, \prod_{i=1}^{n-1}{w_{i}}_{|T_{n}}^{k_{i+1}} ) \in E_{n+1}
  \end{aligned}  
\]
By Proposition~\ref{prop:ncomp}, $\prod_{i=1}^{n-1}{w_{i}}_{|T_{n}}^{k_{i+1}}$ is in $E_{n}$. By the induction hypothesis $k_{i+1}=0$ for all $i=1,\ldots,n-2$ and $\varphi_{n+1}(k_1,\ldots,k_{n})={w_{1}}_{|T_{n+1}}^{k_1}$. 
We are left to show that $k_1=0$. This follows since ${w_{1}}_{|T_{n}}$ is not in $E_n$. Since Proposition~\ref{prop:ncomp} holds for both $E$ and $U$, the proof for the injectivity of $\phi_n$ follows the same argument.
\end{proof}

\begin{cor}\label{cor:injective} 
 Let $k_j\in\F_2$ for $1\leq j \leq n$. If the product $\prod_{j=1}^{n-1}{w_j}^{k_j}$ (resp., $\prod_{j=1}^{n}{w_j}^{k_j}$) is in $E_n$ (resp., in $U_n$), then $k_j=0$ for all $j$.
\end{cor}

\begin{proposition}
 Let $n\geq 2$. The homomorphism $\varphi_n$ (resp., $\phi_n$) is surjective for any $n\geq 2$ (resp., for any $n\geq 1$).
 \end{proposition}

\begin{proof} 
The surjectivity of $\varphi_{2}$ follows from the fact that $[W_2:E_2]=2$ and that $\varphi_2$ is injective. Similarly $[W_1:U_1]=2$ and $\phi_1$ is injective implies that $\phi_1$ is surjective. 

Assume $\varphi_{n}$ is surjective. Let $x=(x_1,\ldots,x_d)\tau$ be in $N(E)$ (resp. $N(U)$). If $\tau_{|T_1}$ is an even permutation, then $\tau$ is in $E$. Otherwise $\tau w^{-1}_{|T_1}$ is an even permutation and again $\tau w^{-1}$ is in $E$. A similar argument works for $U$. 
Hence $\tau={w_1}^k \pmod {E}$ (resp., $\pmod U$) for some $k \in \F_2$. 
So we may assume $x=(x_1,\ldots,x_d)$. By Proposition~\ref{prop:ncomp}, ${x_i}_{|T_n}$ is in $N(E)_{n}$ for all $1\leq i \leq d$. Since $\varphi_n$ is surjective, we have ${x_i}_{|T_n}= \prod_{j=1}^{n-1}{w_j}_{|T_{n}}^{k_{i,j}}g_i$ for some $k_{i,j} \in \F_2$ and $g_i \in E_{n}$ for all $1\leq i \leq d$. Hence ${x}_{|T_{n+1}}$ is
\[ \begin{aligned}
 {x}_{|T_{n+1}}=&(\prod_{j=1}^{n-1}{w_j}_{|T_{n}}^{k_{1,j}},\ldots,\prod_{j=1}^{n-1}{w_j}_{|T_{n}}^{k_{d,j}})(g_1,\ldots,g_d)  \\
 =& (\prod_{j=1}^{n-1}{w_j}_{|T_{n}}^{k_{1,j}},\ldots,\prod_{j=1}^{n-1}{w_j}_{|T_{n}}^{k_{d,j}}){w_1}_{|T_{n+1}}^{k} \pmod {E_{n+1}}
\end{aligned}
\]
for some $k \in \F_2$. By Lemma~\ref{lem:w}, ${w_1}_{|T_{n+1}}^{k} \in N(E)_{n+1}$ and hence  $ (\prod_{j=1}^{n-1}{w_j}_{|T_{n}}^{k_{1,j}},\ldots,\prod_{j=1}^{n-1}{w_j}_{|T_{n}}^{k_{d,j}}) \in N(E)_{n+1}$ and
by Proposition~\ref{prop:ncomp}(2), we find that
\[
\prod_{j=1}^{n-1}{w_j}_{|T_{n}}^{k_{1,j}-k_{i,j}}
\]
is in $E_{n}$ for all $i$. Hence by Corollary~\ref{cor:injective}, $k_{1,j}=k_{2,j}=\ldots=k_{d,j}$ for all $1\leq j \leq n-1$. Let this number be $k_j$ for each $j=1,\ldots,n-1$. Now we have
\[ 
\begin{aligned}
   x_{|T_{n+1}}=&(\prod_{j=1}^{n-1}{w_j}_{|T_{n}}^{k_{j}}, \ldots, \prod_{j=1}^{n-1}{w_j}_{|T_{n}}^{k_{j}}){w_1}_{|T_{n+1}}^{k} & \pmod {E_{n+1}}\\
                           =&\prod_{j=1}^{n-1}({w_j}_{|T_{n}}^{k_{j}},\ldots, {w_j}_{|T_{n}}^{k_{j}}){w_1}_{|T_{n+1}}^{k} & \pmod {E_{n+1}} \\
                            =& {w_1}_{|T_{n+1}}^{k} \prod_{j=1}^{n-1}{w_{j+1}}_{|T_{n+1}}^{k_{j}}& \pmod {E_{n+1}}
\end{aligned}
\]
Hence $\varphi_n$ is surjective. The proof is same for $\phi_n$ since Proposition~\ref{prop:ncomp}  and Lemma~\ref{lem:w} holds for $U$ as well.
\end{proof}

\begin{cor}\label{cor:isom} 
Let $N(E)$ denote the normalizer of $E$ in $W$ and let $N(U)$ denote the normalizer of $U$ in $W$. Then 
\[ \prod_{i=1}^{\infty} \mathbb{F}_2  \to  N(E)/E \text{ and } \prod_{i}^{\infty} \mathbb{F}_2  \to  N(U)/U\] 
given by $(k_1,\ldots,k_i,\ldots) \mapsto w_1^{k_1} w_2^{k_2}\ldots w_n^{k_n}\ldots$
is an isomorphism.

\end{cor}

\section{Arithmetic Monodromy Groups of Dynamical Belyi Maps}\label{sec:6}
In this section $f$ denotes a dynamical Belyi map of combinatorial type $\underline{C}=(d; e_1, e_2,
e_3)\notin\{(4;3,3,3), (6;4,4,5)\}$. To ease the notation we will drop $f$ from the notation and denote $ G^{\text{geom}}(f)$ by $G$. Remember that $G$ is either $E$ or $U$ (Definition~\ref{def:En}) by Theorem~\ref{thm:bek}.Now we are ready to prove the main theorem. 

\begin{theorem} \label{thm:main}
Let $f$ be a dynamical Belyi map of combinatorial type $\underline{C}=(d; e_1, e_2,
e_3) \\ \notin\{(4;3,3,3), (6;4,4,5)\}$. The order of the quotient $G^{\text{arith}}(f)/G^{\text{geom}}(f)$ divides $2$. 
\end{theorem}

\begin{proof}
The idea of the proof comes from \cite[Lemma~4.8.4]{Pinkrational}. 
Using ~\eqref{eq:exact},  we have 
\[ 
\pi_1^{\acute{e}t}(\PP_k^1 \backslash P, x_0)/\pi_1^{\acute{e}t}(\PP_{\bar{k}}^1 \backslash P, x_0) \to G^{\text{arith}}/ G. 
\] 
and that
$\pi_1^{\acute{e}t}(\PP_k^1 \backslash P, x_0)/\pi_1^{\acute{e}t}(\PP_{\bar{k}}^1 \backslash P, x_0)$ is isomorphic to $\Gal(\bar{k}/k)$. Using \eqref{diagram:Gar} and ~\eqref{diagram:Ggeom}, we obtain the following diagram.
\begin{equation}\label{eq:quo}
   \begin{tikzcd}
\pi_1^{\acute{e}t}(\PP_k^1 \backslash P, x_0)/\pi_1^{\acute{e}t}(\PP_{\bar{k}}^1 \backslash P, x_0)      \arrow{r}                      & G^{\text{arith}}/ G                                                 \\
F_k/F_{\bar{k}} \arrow{r}   \arrow{u} \arrow{d}      & G^{\text{arith}} \cap W^d / G \cap W^d \arrow{d}  \arrow[hookrightarrow]{u} \\
 \pi_1^{\acute{e}t}(\PP_k^1 \backslash P, x_0)/\pi_1^{\acute{e}t}(\PP_{\bar{k}}^1 \backslash P, x_0)  \arrow{r}                                      &   G^{\text{arith}}/G   \\
     \end{tikzcd}
  \end{equation}
The quotient $F_k/F_{\bar{k}}$ is $ \{(\gamma_1 ,\ldots,\gamma_d) \in \prod_{i=1}^d \Gal(\bar{k}/k) :\gamma_1=\ldots=\gamma_d\}$. Putting all of these discussions together, we obtain the diagram below.
  \begin{equation}\label{eq:quo}
   \begin{tikzcd}
\Gal(\bar{k}/ k) \arrow{r}                      & G^{\text{arith}}/ G                                                 \\
\Gal(\bar{k}/ k) \arrow{r}   \ar[equal]{u}  \ar[equal]{d}     & G^{\text{arith}} \cap W^d / G \cap W^d \arrow{d}  \arrow[hookrightarrow]{u} \\
  \Gal(\bar{k}/ k)   \arrow{r}                                      &   G^{\text{arith}}/G   \\
     \end{tikzcd}
  \end{equation}
Since $G$ is a normal subgroup of $G^{\text{arith}}$, the quotient $G^{\text{arith}}/G$ is a subgroup of $N(G)/ G$ which we studied in detail in Corollary~\ref{cor:isom}. We first note that by Corollary~\ref{cor:isom}, $N(G)/ G$ is isomorphic to a direct product of copies of $\F_2$ and hence the exact sequence 
\[
1 \to N(G)\cap W^d / G \cap W^d \to N(G)/G \to \langle {w_1}_{|T_1}\rangle \to 1.
\]
splits with trivial action. Hence 
\[  
N(G)/G \simeq (N(G)\cap W^d / G \cap W^d) \times \langle {w_1}_{|T_1}\rangle.
\]
We note that $w_i$ is in $N(G)\cap W^d$ for all $i\geq 2$. Therefore we have a projection map 
\[ 
N(G)/G \to N(G)\cap W^d / G \cap W^d
 \] sending the automorphism $\prod_{i=1}^{\infty}w_i^{k_i} \pmod {G} $ to $\prod_{i=2}^{\infty}w_i^{k_{i}}  \pmod{G}$. Furthermore we can compose it with the projection of $W^d$ onto the first component. The composition of these two maps gives a homomorphism 
$ N(G)/G \to N(G)/ G$ that maps
\begin{equation}\label{eq:visom}
w_1^{k_1}w_2^{k_2}\ldots =w_1^{k_1}(\prod_{i=1}^{\infty}w_i^{k_{i+1}},\ldots,\prod_{i=1}^{\infty}w_i^{k_{i+1}})\mapsto \prod_{i=1}^{\infty}w_i^{k_{i+1}} \end{equation}
since $w_i=(w_{i-1},\ldots,w_{i-1})$ for $i\geq 2$.

We can see these maps in the following diagram. 

 \begin{equation}\label{eq:exact2}
   \begin{tikzcd}
\Gal(\bar{k}/ k) \arrow{r} \ar[equal]{dd} & G^{\text{arith}}/ G \arrow{r} \arrow{d}
                                       & N(G)/ G \arrow{r}{\simeq} \arrow{d} 
                                       &  \prod_{i}^{\infty} \mathbb{F}_2 \arrow{dd}  \\& G^{\text{arith}} \cap W^d / G \cap W^d \arrow{r} \arrow{d} & N(G)\cap W^d / G \cap W^d   \arrow{d} &                                                                                                                      \\
  \Gal(\bar{k}/ k)   \arrow{r} &   G^{\text{arith}}/G \arrow{r}  &N(G)/ G  \arrow{r}{\simeq} &  \prod_{i}^{\infty} \mathbb{F}_2 \\
     \end{tikzcd}
  \end{equation}
  
We obtain the left part of the diagram from ~\eqref{eq:quo}. We define the vertical map on the right hand side as $(k_1,k_2, \ldots) \to (k_2,k_3,\ldots)$ so that the diagram commutes.

Let  $w=\prod_{i=1}^{\infty}w_i^{k_i} \pmod {G}$ be an element of $G^{\text{arith}} /G$. Since $\Gal(\bar{k}/ k) \to G^{\text{arith}} /G$ is surjective, there is a $\tau \in \Gal(\bar{k}/k)$ that maps to $w$. The image of $\tau$ in the top row is $(k_1,k_2,\ldots)$ and in the bottom row $(k_2,k_3,\ldots)$. Since the diagram commutes, they are equal and $k_1=k_2=\ldots=k_n$ for all $n$ which shows that $G^{\text{arith}}/G$ is a subset of $\{(1,1,1,\ldots), (0,0,0,\ldots)\}$. 

\end{proof}

\begin{remark}
In \cite[Lemma~2.4.3]{BEK}, a product discriminant $D$ is defined. Moreover, it is shown that $G^{\text{arith}}=G^{\text{geom}}$ if and only if $D$ is a square in $\Q(t)$. One can generalize this to $k(t)$ where $k$ is a number field. By Theorem~\ref{thm:main}, we know that $\Gal(\bar{k}/ k) \to G^{\text{arith}} /G$ factors through a quadratic extension $L$ of $k$. Furthermore, in \cite[Proposition~2.4.5]{BEK}, this discriminant $D$ is explicitly calculated. From this we see that the quadratic extension $L$ is given by $k(\sqrt{u})$ where $D=u(1-t)^{2(e_2-1)}t^{2(e_1-1)}$ for a dynamical Belyi map of combinatorial type $(d;e_1,e_2,e_3)$.
\end{remark}


\end{document}